\documentclass[a4paper]{amsart}
\usepackage{amsmath,amsthm,amssymb,latexsym,epic,bbm,comment}
\usepackage{graphicx,enumerate,stmaryrd}
\usepackage[all]{xy}
\xyoption{2cell}

\usepackage[active]{srcltx}
\usepackage[parfill]{parskip}

\newtheorem{theorem}{Theorem}
\newtheorem{lemma}[theorem]{Lemma}

\newtheorem{proposition}[theorem]{Proposition}

\newcommand{\tto}{\twoheadrightarrow}

\font\sc=rsfs10

\newcommand{\cG}{\sc\mbox{G}\hspace{1.0pt}}

\begin{document}

\title[Categorification of the Catalan monoid]{Categorification of the Catalan monoid}
\author{Anna-Louise Grensing and Volodymyr Mazorchuk}

\begin{abstract}
We construct a finitary additive $2$-category whose Grothendieck ring is isomorphic to the
semigroup algebra of the monoid of order-decreasing and order-preserving transformations of
a finite chain.
\end{abstract}

\maketitle

\begin{center}
{\em Dedicated to the memory of John Howie}
\end{center}

\section{Introduction and description of the results}\label{s1}

For a positive integer $n$ consider the set $\mathbf{N}_n:=\{1,2,\dots,n\}$ which is linearly ordered
in the usual way. Let $\mathcal{T}_n$ be the full transformation monoid on $\mathbf{N}_n$, that is the
set of all total maps $f:\mathbf{N}_n\to \mathbf{N}_n$ with respect to composition
(from right to left). Let $\mathcal{C}_n$ denote the submonoid of $\mathcal{T}_n$ consisting of all 
maps which are:
\begin{itemize}
\item {\em order-decreasing} in the sense that $f(i)\leq i$ for all $i\in \mathbf{N}_n$;
\item {\em order-preserving} in the sense that $f(i)\leq f(j)$ for all $i,j\in \mathbf{N}_n$ such
that $i\leq j$.
\end{itemize}
Elements of $\mathcal{C}_n$ are in a natural bijection with lattice paths from $(0,0)$ to $(n,n)$ which
remain below the diagonal. The bijection is given by sending $f\in\mathcal{C}_n$ to the path in which 
for every $i\in\{1,2,\dots,n-1\}$ the maximal $y$-coordinate for a path point having the $x$-coordinate $i$
equals $f(i+1)-1$. Hence $|\mathcal{C}_n|=\mathrm{C}_n:=\frac{1}{n+1}\binom{2n}{n}$ is 
$n$-th Catalan number (see \cite{Hi}).
Because of this, $\mathcal{C}_n$ is usually called the {\em Catalan monoid},
see \cite{So,MS}, (some other names are ``the monoid of non-decreasing parking functions'' or, simply, 
``the monoid of order-decreasing and order-preserving'' maps). The monoid $\mathcal{C}_n$ is a classical 
object of combinatorial semigroup theory, see e.g. \cite{Hi,Ho,So,LU,GM,DHST,AAK} and references therein.

For $i=1,2,\dots,n-1$ let $\alpha_i$ denote the element of $\mathcal{C}_n$ defined as follows:
\begin{displaymath}
\alpha_i(j):=\begin{cases}j-1,& j=i+1;\\j,&\text{otherwise}.\end{cases}
\end{displaymath}
It is easy to check that the $\alpha_i$'s satisfy the following relations:
\begin{equation}\label{eq1}
\alpha_i^2=\alpha_i;\quad\quad
\alpha_i\alpha_{i+1}\alpha_i=\alpha_{i+1}\alpha_i\alpha_{i+1}=\alpha_i\alpha_{i+1};\quad\quad
\alpha_i\alpha_j=\alpha_j\alpha_i, |i-j|>1.
\end{equation}
Moreover, in \cite{So} it is shown that this gives a presentation for $\mathcal{C}_n$
(see also \cite{GM2} for a short argument). This means that 
$\mathcal{C}_n$ is a {\em Kiselman quotient} of the {\em $0$-Hecke monoid} of type $A_{n\text{-}1}$ as defined in
\cite{GM2}. The middle relation in \eqref{eq1} is the defining relation for {\em Kiselman semigroups},
see \cite{KM}.

The combinatorial datum defining a Kiselman quotient $\mathbf{KH}_{\Gamma}$ of a $0$-Hecke monoid 
$\mathbf{H}_{\Gamma}$ is given by a finite quiver $\Gamma$. In the case when $\Gamma$ does not 
contain oriented cycles, it was shown in \cite{Pa,Gr} that there is a weak functorial action of 
$\mathbf{KH}_{\Gamma}$ on the category of modules of the path algebra of $\Gamma$, i.e. there exist endofunctors on
this module category which satisfy the defining relations of $\mathbf{KH}_{\Gamma}$ (up to isomorphism
of functors). These are the so-called {\em projection functors} associated to simple modules of the path algebra. In the special case of $\mathcal{C}_n$, we obtain a weak action of $\mathcal{C}_n$ on the category of modules over the path algebra of the following quiver, which we will denote by $\mathbf{Q}=\mathbf{Q}_{n-1}$
(note that the set of vertices of $\mathbf{Q}$ coincides with $\mathbf{N}_{n-1}$):

\begin{equation}\label{eq2}
\xymatrix{
\mathtt{1}\ar[rr]&&\mathtt{2}\ar[rr]&&\mathtt{3}\ar[rr]&&\dots\ar[rr]&&\mathtt{n}\text{-}\mathtt{1}
}
\end{equation}

In the present paper we further develop this idea putting it into the general context of  algebraic 
{\em categorification} as described in e.g. \cite{Ma}. 
Let $\Gamma$ be a finite acyclic quiver. Denote by $\mathcal{A}_{\Gamma}$ the 
$\mathbb{Z}$-linear path category of $\Gamma$ (i.e. objects in $\mathcal{A}_{{\Gamma}}$ are vertices of 
$\Gamma$, morphisms in $\mathcal{A}_{\Gamma}$ are formal $\mathbb{Z}$-linear spans of oriented 
paths in $\Gamma$ and composition is given by concatenation of paths). For an algebraically closed field  
$\Bbbk$ we denote by $\mathcal{A}_{\Gamma}^{\Bbbk}$
the $\Bbbk$-linear version of $\mathcal{A}_{\Gamma}$, that is the version in which scalars are extended to $\Bbbk$. Consider the category 
$\mathcal{A}_{\Gamma}^{\Bbbk}\text{-}\mathrm{mod}$  of left finite dimensional $\mathcal{A}_{\Gamma}^{\Bbbk}$-modules.  (It is equivalent to the category of finite-dimensional $\Gamma$-representations over $\Bbbk$ and we will not distinguish between them.)
Let $\mathcal{X}$ be some small category equivalent to $\mathcal{A}_{\Gamma}^{\Bbbk}\text{-}\mathrm{mod}$ 
(we fix an equivalence between these two categories which allows us to consider all endofunctors 
of $\mathcal{A}_{\Gamma}^{\Bbbk}\text{-}\mathrm{mod}$ as endofunctors of $\mathcal{X}$). Projection functors 
preserve the category of injective $\mathcal{A}_{\Gamma}^{\Bbbk}$-modules. Using the action of projection functors on 
the category of injective modules, we define certain endofunctors $\mathrm{G}_i$
of $\mathcal{A}_{\Gamma}^{\Bbbk}\text{-}\mathrm{mod}$ which turn out to be {\em exact}.
These endofunctors are used to define a finitary and 
additive $2$-ca\-te\-gory $\cG_{\Gamma}$ as follows: The $2$-category $\cG_{\Gamma}$ has one object which we identify with 
$\mathcal{X}$. The set of $1$-morphisms in $\cG_{\Gamma}$ consists of all endofunctors of $\mathcal{X}$, which are 
isomorphic to a direct sum of compositions of the 
$\mathrm{G}_i$'s. The set of $2$-morphisms between any pair 
of $1$-morphisms is given by all natural transformations of functors. For simplicity, set $\cG_{n}:=\cG_{\mathbf{Q}}$. Our main result is the following
claim which reveals a nice interplay between semigroup theory, representation theory, category theory and
combinatorics:

\begin{theorem}\label{thmmain}
Denote by $\mathcal{S}[\cG_n]$ the set of isomorphism classes  of indecomposable $1$-morphisms in $\cG_n$.
\begin{enumerate}[$($a$)$]
\item\label{thmmain.1} Composition of $1$-morphisms defines on $\mathcal{S}[\cG_n]$ the structure of
a semigroup.
\item\label{thmmain.2} There is an isomorphism $\Phi:\mathcal{C}_n\to \mathcal{S}[\cG_n]$ of monoids.
\item\label{thmmain.3} The morphism space of the Grothendieck category $\mathrm{Gr}(\cG_n)$ 
is isomorphic to the integral group algebra $\mathbb{Z}[\mathcal{C}_n]$.
\item\label{thmmain.4} The action of $\mathrm{Gr}(\cG_n)$ on the Grothendieck group
$\mathrm{Gr}(\mathcal{X})$ gives rise to a linear representation of $\mathcal{C}_n$ as constructed 
in \cite[Subsection~3.2]{GM2}.
\end{enumerate}
\end{theorem}

Theorem~\ref{thmmain} says that $\cG_n$ is a genuine categorification of $\mathcal{C}_n$ in the sense of
\cite{CR,Ro1,Ma}. We refer the reader to \cite{Ro2,MM1,MM2} for further examples of categorification.
In Section~\ref{s2} we give all definitions and constructions with all necessary details.
Theorem~\ref{thmmain} is proved in Section~\ref{s3}. In Section~\ref{s4} we discuss various consequences
of Theorem~\ref{thmmain} and study some related questions. Finally, in Section~\ref{s5} we 
describe the category of $2$-morphisms and construct cell $2$-representations of $\cG_n$.
\vspace{2mm}

\noindent
{\bf Acknowledgment.} A substantial part of the paper was written during a visit of the first author to 
Uppsala University, whose  hospitality is gratefully acknowledged. The visit was supported by the Swedish 
Research Council and the Department of Mathematics, Uppsala University. The second author is partially 
supported by the  Swedish Research Council and the Royal Swedish Academy of Sciences. We thank the referee
for many helpful comments.

\section{The $2$-category $\cG_{\Gamma}$}\label{s2}

In what follows $\Gamma$ denotes a finite acyclic quiver.

\subsection{$\mathcal{A}_{\Gamma}^{\Bbbk}$-modules}\label{s2.1}

A finite dimensional left $\mathcal{A}_{\Gamma}^{\Bbbk}$-module is a covariant $\Bbbk$-linear functor from  
$\mathcal{A}_{\Gamma}^{\Bbbk}$ to $\Bbbk\text{-}\mathrm{mod}$, the category of finite dimensional vector 
spaces over $\Bbbk$. Morphisms of $\mathcal{A}_{\Gamma}^{\Bbbk}$-modules are natural transformations of functors.

For $\mathtt{i}\in\Gamma$ we denote by $P_{\mathtt{i}}$ the indecomposable projective module 
$\mathcal{A}_{\Gamma}^{\Bbbk}(\mathtt{i},{}_-)$. For $\mathtt{j}\in\Gamma$ the value $P_{\mathtt{i}}(\mathtt{j})$ 
is thus the $\Bbbk$-vector space $\mathcal{A}_{\Gamma}^{\Bbbk}(\mathtt{i},\mathtt{j})$ and for a morphism
$f\in \mathcal{A}_{\Gamma}^{\Bbbk}(\mathtt{j},\mathtt{j}')$ we have $P_{\mathtt{i}}(f):=f\circ{}_-$.

We denote by $L_{\mathtt{i}}$ the simple top of $P_{\mathtt{i}}$. This can be understood as follows: 
$L_{\mathtt{i}}(\mathtt{j})$ is zero unless ${\mathtt{i}}={\mathtt{j}}$ and in the latter case 
$L_{\mathtt{i}}(\mathtt{i})=\Bbbk$. All paths except for the identity path at $\mathtt{i}$
are sent by $L_{\mathtt{i}}$ to zero maps. Finally, denote by $I_{\mathtt{i}}$ the indecomposable injective 
envelope of $L_{\mathtt{i}}$. For ${\mathtt{j}}\in\Gamma$ the value $I_{\mathtt{i}}(\mathtt{j})$ is thus 
the $\Bbbk$-vector space $\mathcal{A}_{\Gamma}^{\Bbbk}(\mathtt{j},\mathtt{i})^*$ 
(here $*$ denotes the $\Bbbk$-dual vector space) and for 
$f\in \mathcal{A}_{\Gamma}^{\Bbbk}(\mathtt{j},\mathtt{j}')$ 
we have $I_{\mathtt{i}}(f):=f^*\circ{}_-$.

\subsection{Projection functors}\label{s2.2}

For two modules $M$ and $N$ the {\em trace} $\mathrm{Tr}_M(N)$ of $M$ in $N$ is defined as the sum of images of 
all homomorphisms from $M$ to $N$. If $\varphi:N\to N'$ is a homomorphism, then 
$\varphi$ maps elements from $\mathrm{Tr}_M(N)$
to elements from $\mathrm{Tr}_M(N')$. This means that $\mathrm{Tr}_M$ is naturally a subfunctor of
the identity endofunctor $\mathrm{Id}_{\mathcal{A}_{\Gamma}^{\Bbbk}\text{-}\mathrm{mod}}$.

For $\mathtt{i}\in\Gamma$ the corresponding {\em projection endofunctor} $\mathrm{F}_{\mathtt{i}}$ of
$\mathcal{A}_{\Gamma}^{\Bbbk}\text{-}\mathrm{mod}$ is defined as 
$\mathrm{Id}_{\mathcal{A}_{\Gamma}^{\Bbbk}\text{-}\mathrm{mod}}/\mathrm{Tr}_{L_{\mathtt{i}}}$. As shown in \cite{Pa,Gr},
projection endofunctors preserve both monomorphisms and epimorphisms but they are neither left nor right exact.
The latter property is a problem as it means that projection functors do not induce, in any natural way, any
linear action on the Grothendieck group of $\mathcal{A}_{\Gamma}^{\Bbbk}\text{-}\mathrm{mod}$. To fix this problem
we will consider a slight variation of projection functors.

\subsection{Partial approximations}\label{s2.3}

Note that $\mathcal{A}_{\Gamma}^{\Bbbk}\text{-}\mathrm{mod}$ is a hereditary category and hence any quotient
of an injective object is injective. As projection functors are quotients of the identity functor, it follows
that all projection functors preserve $\mathcal{I}_{\mathcal{A}_{\Gamma}^{\Bbbk}}$, the category of injective
$\mathcal{A}_{\Gamma}^{\Bbbk}$-modules. Consider 
the $\mathcal{A}_{\Gamma}^{\Bbbk}\text{-}\mathcal{A}_{\Gamma}^{\Bbbk}$-bimodule
$P:=\mathcal{A}_{\Gamma}^{\Bbbk}({}_-,{}_-)$. As a left module, $P$ is a projective generator 
of $\mathcal{A}_{\Gamma}^{\Bbbk}\text{-}\mathrm{mod}$.
Let $I:=P^*$ be the corresponding dual bimodule, the injective cogenerator of 
$\mathcal{A}_{\Gamma}^{\Bbbk}\text{-}\mathrm{mod}$. Define
\begin{displaymath}
\mathrm{G}_{\mathtt{i}}:=\mathrm{Hom}_{\mathcal{A}_{\Gamma}^{\Bbbk}}((\mathrm{F}_{\mathtt{i}}\,I)^* ,{}_-).
\end{displaymath}
Then $\mathrm{G}_{\mathtt{i}}$ is a left exact endofunctor of $\mathcal{A}_{\Gamma}^{\Bbbk}\text{-}\mathrm{mod}$.
Furthermore, using adjunction we have the following natural isomorphism:
\begin{displaymath}
\begin{array}{rcl}
\mathrm{G}_{\mathtt{i}}\, I&=&
\mathrm{Hom}_{\mathcal{A}_{\Gamma}^{\Bbbk}}((\mathrm{F}_{\mathtt{i}}\,I)^* ,\mathrm{Hom}_{\Bbbk}(P,\Bbbk))\\
&\cong&\mathrm{Hom}_{\Bbbk}(P\otimes_{\mathcal{A}_{\Gamma}^{\Bbbk}}(\mathrm{F}_{\mathtt{i}}\,I)^* ,\Bbbk)\\
&\cong&\mathrm{Hom}_{\Bbbk}((\mathrm{F}_{\mathtt{i}}\,I)^* ,\Bbbk)\\
&\cong&\mathrm{F}_{\mathtt{i}}\,I,\\
\end{array}
\end{displaymath}
which implies that $\mathrm{G}_{\mathtt{i}}\vert_{\mathcal{I}_{\mathcal{A}_{\Gamma}^{\Bbbk}}}\cong
\mathrm{F}_{\mathtt{i}}\vert_{\mathcal{I}_{\mathcal{A}_{\Gamma}^{\Bbbk}}}$ as $I$ generates 
$\mathcal{I}_{\mathcal{A}_{\Gamma}^{\Bbbk}}$ additively.

The functors $\mathrm{G}_{\mathtt{i}}$ have an alternative realization as partial approximation functors 
considered in \cite{KhMa}. Let $Q_{\mathtt{i}}:=\oplus_{{\mathtt{j}}\neq {\mathtt{i}}}I_{\mathtt{j}}$. 
Given an $\mathcal{A}_{\Gamma}^{\Bbbk}$-module $M$, let $M\hookrightarrow I_M$ be 
some injective envelope of $M$. Denote by $M'$ the intersection of kernels of all morphisms $f:I_M\to Q_{\mathtt{i}}$
satisfying $f(M)=0$. Denote by $M''$ the intersection of kernels of all morphisms $M'\to Q_{\mathtt{i}}$. 
Then the correspondence $M\mapsto M'/M''$ with the natural action on morphisms is functorial (see \cite{KhMa}). 
This functor is called the {\em partial approximation} $\mathrm{H}_{Q_{\mathtt{i}}}$ with respect to 
$Q_{\mathtt{i}}$ and it comes together with a natural  transformation  
$\mathrm{Id}_{\mathcal{A}_{\Gamma}^{\Bbbk}\text{-}\mathrm{mod}}\to \mathrm{H}_{Q_{\mathtt{i}}}$ which
is injective on all submodules of $Q_{\mathtt{i}}^{\oplus k}$. Then \cite[Comparison Lemma]{KhMa} implies that 
$\mathrm{H}_{Q_{\mathtt{i}}}$ and $\mathrm{G}_{\mathtt{i}}$ are isomorphic.
The functors $\mathrm{G}_{\mathtt{i}}$ appear in \cite{Pa} under the
name ``orthogonal functors''.

\subsection{Exactness}\label{s2.6}

A surprising property of the functor $\mathrm{G}_{\mathtt{i}}$ is the following (compare with approximation functors
from \cite{KhMa} and also with Subsection~\ref{s4.5}):

\begin{proposition}\label{exact}
The functor $\mathrm{G}_{\mathtt{i}}$ is exact for every ${\mathtt{i}}\in\Gamma$.
\end{proposition}

\begin{proof}
By construction,  $\mathrm{G}_{\mathtt{i}}$ is left exact. As $\mathcal{A}_{\Gamma}^{\Bbbk}\text{-}\mathrm{mod}$ is 
hereditary, to prove exactness of $\mathrm{G}_{\mathtt{i}}$ we only have to show that $\mathcal{R}^1\mathrm{G}_{\mathtt{i}}=0$.
Let $M\in \mathcal{A}_{\Gamma}^{\Bbbk}\text{-}\mathrm{mod}$ and $M\hookrightarrow Q_0\tto Q_1$ be an
injective coresolution of $M$. As $\mathrm{G}_{\mathtt{i}}\vert_{\mathcal{I}_{\mathcal{A}_{\Gamma}^{\Bbbk}}}\cong
\mathrm{F}_{\mathtt{i}}\vert_{\mathcal{I}_{\mathcal{A}_{\Gamma}^{\Bbbk}}}$, the module 
$\mathcal{R}^1\mathrm{G}_{\mathtt{i}}\, M$ is isomorphic to the homology of
the sequence $\mathrm{F}_{\mathtt{i}}\, Q_0\to \mathrm{F}_{\mathtt{i}}\, Q_1\to 0$. Since $\mathrm{F}_{\mathtt{i}}$ preserves surjections
(see \cite[Page~9]{Pa}), this homology is zero. The claim follows.
\end{proof}

\subsection{Definition of $\cG_{\Gamma}$}\label{s2.5}

Let $\mathcal{X}$ be some small additive category equivalent to $\mathcal{A}_{\Gamma}^{\Bbbk}\text{-}\mathrm{mod}$.
Fixing such an equivalence we can define the action of $\mathrm{G}_{\mathtt{i}}$ 
(and also of $\mathrm{F}_{\mathtt{i}}$) on $\mathcal{X}$ up to isomorphism of functors. Consider
the $2$-category $\cG_{\Gamma}$ defined as follows: $\cG_{\Gamma}$ has one object, which we identify with $\mathcal{X}$;
$1$-morphisms of $\cG_{\Gamma}$ are all endofunctors of $\mathcal{X}$ which are isomorphic to direct sums of 
functors, each of which is isomorphic to a direct summand in some composition of functors in which 
every factor is isomorphic to some $\mathrm{G}_{\mathtt{i}}$; $2$-morphisms of $\cG_{\Gamma}$ are natural 
transformations of functors. Note that all $\Bbbk$-spaces of $2$-morphisms are finite dimensional.
Furthermore, by definition, the category $\cG_{\Gamma}(\mathcal{X},\mathcal{X})$ is fully additive
and $\cG_{\Gamma}$ is enriched over the category of additive $\Bbbk$-linear categories. Set $\cG_{n}:=\cG_{\mathbf{Q}}$.

\subsection{The $2$-action on the derived category}\label{s2.4}

That the functor $\mathrm{G}_{\mathtt{i}}$ turned out to be exact is accidental. In the general case
(of a non-hereditary algebra) one can only expect $\mathrm{G}_{\mathtt{i}}$ to be left exact. However, 
as the following alternative description shows, this is not a big problem.
Denote by $\mathcal{X}_{\mathcal{I}}$ the full subcategory of injective objects in $\mathcal{X}$.
Let $\mathcal{K}^b(\mathcal{X}_{\mathcal{I}})$ be the bounded homotopy category of the additive category
$\mathcal{X}_{\mathcal{I}}$. Then we have a natural $2$-action of $\cG_{\Gamma}$ on 
$\mathcal{K}^b(\mathcal{X}_{\mathcal{I}})$ defined componentwise.

Note that $\mathcal{K}^b(\mathcal{X}_{\mathcal{I}})$ is a triangulated category which is equivalent
to the bounded derived category $\mathcal{D}^b(\mathcal{A}_{\Gamma}^{\Bbbk})$ of the abelian category 
$\mathcal{A}_{\Gamma}^{\Bbbk}\text{-}\mathrm{mod}$. Via this equivalence the 
action of $\mathrm{G}_{\mathtt{i}}$ on $\mathcal{K}^b(\mathcal{X}_{\mathcal{I}})$ can thus be considered as an action on
$\mathcal{D}^b(\mathcal{A}_{\Gamma}^{\Bbbk})$ and this is exactly the definition of the right derived
functor $\mathcal{R}\mathrm{G}_{\mathtt{i}}$.

For $\mathbf{i}:=({\mathtt{i}}_1,{\mathtt{i}}_2,\dots,{\mathtt{i}}_k)\in\Gamma^k$ consider the compositions
\begin{displaymath}
\mathrm{G}_{\mathbf{i}}:=\mathrm{G}_{{\mathtt{i}}_1}\circ\mathrm{G}_{{\mathtt{i}}_2}\circ\dots\circ
\mathrm{G}_{{\mathtt{i}}_k}\quad\text{ and }\quad
(\mathcal{R}\mathrm{G})_{\mathbf{i}}:=\mathcal{R}\mathrm{G}_{{\mathtt{i}}_1}\circ
\mathcal{R}\mathrm{G}_{{\mathtt{i}}_2}\circ\dots
\circ\mathcal{R}\mathrm{G}_{{\mathtt{i}}_k}.
\end{displaymath}

\begin{lemma}\label{lem1}
There is an isomorphism of functors as follows:
$(\mathcal{R}\mathrm{G})_{\mathbf{i}}\cong\mathcal{R}(\mathrm{G}_{\mathbf{i}})$.
\end{lemma}

\begin{proof}
This follows directly from exactness of the $\mathrm{G}_{\mathtt{i}}$'s. 
\end{proof}

In the more general (non-hereditary) case to prove Lemma~\ref{lem1} one could use 
the standard spectral sequence argument as described, for example, in \cite[Section~III.7]{GeMa}.

\section{Proof of Theorem~\ref{thmmain}}\label{s3}

\subsection{Indecomposability}\label{s3.1}

Our crucial observation is the following:

\begin{proposition}\label{prop2}
Let $\Gamma=\mathbf{Q}$ and $\mathbf{i}:=(\mathtt{i}_1,\mathtt{i}_2,\dots,\mathtt{i}_k)\in\mathbf{Q}^k$. Then 
$\mathrm{G}_{\mathbf{i}}$ is either an indecomposable functor or zero.
\end{proposition}

\begin{proof}
Assume that $\mathrm{G}_{\mathbf{i}}\neq 0$.
As $\mathrm{G}_{\mathbf{i}}$ is left exact, it is enough to show that its restriction to 
$\mathcal{I}_{\mathcal{A}_{\mathbf{Q}}^{\Bbbk}}$ is indecomposable. Consider the indecomposable 
$\mathcal{A}_{\mathbf{Q}}^{\Bbbk}\text{-}\mathcal{A}_{\mathbf{Q}}^{\Bbbk}$-bimodule $I$.
By construction, $\mathrm{G}_{\mathbf{i}}\, I$ is a quotient of $I$. Hence to prove our claim it is enough to show
that every non-zero quotient of the bimodule $I$ is an indecomposable bimodule. For this we have to show that the
bimodule $I$ has simple top. 

As a left module, $I$ is a direct sum of the $I_{\mathtt{i}}$'s. Each $I_{\mathtt{i}}$ has simple top isomorphic 
to $L_{\mathtt{1}}$. The right module structure on $I$ is given by surjective homomorphisms 
$I_{\mathtt{i}}\to I_{\mathtt{i}\text{-}\mathtt{1}}$, each of which sends simple top to simple top (or zero 
if $\mathtt{i}=\mathtt{1}$). This means that the top of the bimodule $I$ coincides with the top of
$I_{\mathtt{n}\text{-}\mathtt{1}}$ and hence is simple. The claim of the proposition follows. 
\end{proof}

\subsection{The (multi)semigroup of $\cG_n$}\label{s3.2}

Denote by $\mathcal{S}[\cG_n]$ the set of isomorphism classes of indecomposable $1$-morphisms in $\cG_n$.
By \cite{MM5}, composition of $1$-mor\-phisms induces on this set a natural structure of a multisemigroup.
In Proposition~\ref{prop2} above it is shown that any composition of projection functors is indecomposable
or zero. This implies that the multisemigroup structure on $\mathcal{S}[\cG_n]$ is, in fact, single valued,
and hence is a semigroup structure.

By construction, $\mathcal{S}[\cG_n]$ is generated by the classes of $\mathrm{G}_{\mathtt{i}}$.
By \cite{Pa,Gr}, these generators satisfy the Hecke-Kiselman relations corresponding to $\mathbf{Q}$.
Therefore $\mathcal{S}[\cG_n]$ is a quotient of the corresponding Hecke-Kiselman semigroup $\mathcal{C}_n$.

\subsection{Decategorification of the defining representation}\label{s3.3}

Consider the Grothendieck category $\mathrm{Gr}(\cG_n)$ of $\cG_n$. This is a usual category with one object
(which we identify with $\mathcal{X}$),
whose endomorphisms are identified with the split Grothendieck group of the fully additive category
$\cG_n(\mathcal{X},\mathcal{X})$. Let $\mathrm{Gr}(\mathcal{X})$ denote the Grothendieck group of the
abelian category $\mathcal{X}$. As every $1$-morphism in $\cG_n(\mathcal{X},\mathcal{X})$ is an exact
endofunctor of $\mathcal{X}$, the defining functorial action of $\cG_n(\mathcal{X},\mathcal{X})$ on $\mathcal{X}$
induces a usual action of the ring $\mathrm{Gr}(\cG_n)(\mathcal{X},\mathcal{X})$ on the abelian group
$\mathrm{Gr}(\mathcal{X})$.

Choose in $\mathrm{Gr}(\mathcal{X})$ a basis $\mathbf{b}_I$ consisting of the classes of indecomposable
injective modules $[I_{\mathtt{1}}]$, $[I_{\mathtt{2}}]$,\dots, 
$[I_{\mathtt{n}\text{-}\mathtt{1}}]$ (in this order). Then a direct calculation shows that the
linear transformation corresponding to the class $[\mathrm{G}_{\mathtt{i}}]$ is given in the basis $\mathbf{b}_I$
by the following $(n\text{-}1)\times (n\text{-}1)$-matrix:
\begin{equation}\label{eq25}
M_i:=\left(
\begin{array}{cccccccccccc}
1&0&0&\dots&0&0&0&0&\dots&0&0&0\\
0&1&0&\dots&0&0&0&0&\dots&0&0&0\\
0&0&1&\dots&0&0&0&0&\dots&0&0&0\\
\vdots&\vdots&\vdots&\ddots&\vdots&\vdots&\vdots&\vdots&\ddots&\vdots&\vdots&\vdots\\
0&0&0&\dots&1&0&0&0&\dots&0&0&0\\
0&0&0&\dots&0&1&1&0&\dots&0&0&0\\
0&0&0&\dots&0&0&0&0&\dots&0&0&0\\
0&0&0&\dots&0&0&0&1&\dots&0&0&0\\
\vdots&\vdots&\vdots&\ddots&\vdots&\vdots&\vdots&\vdots&\ddots&\vdots&\vdots&\vdots\\
0&0&0&\dots&0&0&0&0&\dots&1&0&0\\
0&0&0&\dots&0&0&0&0&\dots&0&1&0\\
0&0&0&\dots&0&0&0&0&\dots&0&0&1
\end{array}
\right)  
\end{equation}
(here the zero row is the $i$-th row from the top). These matrices coincide with the matrices of the 
natural representation of  $\mathcal{C}_n$, see \cite{GM2}, which is known to be effective. This implies that
$\mathcal{S}[\cG_n]\cong \mathcal{C}_n$.

The above establishes claims \eqref{thmmain.1}, \eqref{thmmain.2} and \eqref{thmmain.4} of Theorem~\ref{thmmain}.
Claim \eqref{thmmain.3} follows from claim \eqref{thmmain.2} by taking the integral group rings
on both sides. Note that, from the fact that the semigroup $\mathcal{S}[\cG_n]$ is finite, it now follows that 
the $2$-category $\cG_n$ is finitary in the sense of \cite{MM1}.

\section{Various consequences and related questions}\label{s4}

\subsection{Change of basis}\label{s4.1}

As explained in \cite{Ma}, one of the advantages of the categorical picture above is the fact that the
abelian group $\mathrm{Gr}(\mathcal{X})$ has several natural bases. Above we used the basis given by
classes of indecomposable injective modules. Two other natural bases in $\mathrm{Gr}(\mathcal{X})$ are given by
classes of indecomposable projective modules and by classes of simple modules, respectively. We consider first
the basis of simple modules. 

\begin{lemma}\label{lem21}
For $\mathtt{i},\mathtt{j}\in\mathbf{Q}$ we have
\begin{displaymath}
\mathrm{G}_\mathtt{i}\, L_\mathtt{j}\cong
\begin{cases}
L_\mathtt{j},& \mathtt{j}\neq \mathtt{i},\mathtt{i}+\mathtt{1};\\
X_\mathtt{i},& \mathtt{j}=\mathtt{i}+1;\\
0, & \mathtt{j}=\mathtt{i},
\end{cases}
\end{displaymath}
where $X_\mathtt{i}$ is the unique (up to isomorphism) module for which there is a non-split short exact sequence
as follows: $L_{\mathtt{i}+\mathtt{1}}\hookrightarrow X_\mathtt{i}\tto L_\mathtt{i}$.
\end{lemma}

\begin{proof}
The claim follows by applying $\mathrm{G}_\mathtt{i}$ to the following   injective coresolution of $L_\mathtt{j}$:  
$L_\mathtt{j}\hookrightarrow I_\mathtt{j}\tto I_{\mathtt{j}\text{-}\mathtt{1}}$ (here $I_\mathtt{0}:=0$).
\end{proof}

Choose in $\mathrm{Gr}(\mathcal{X})$ a basis $\mathbf{b}_L$ consisting of the classes of simple 
modules $[L_{{\mathtt{n}}\text{-}\mathtt{1}}]$,\dots, $[L_{\mathtt{2}}]$,$[L_{\mathtt{1}}]$ 
(in this order). From Lemma~\ref{lem21} it follows
that the linear transformation corresponding to the class $[\mathrm{G}_\mathtt{i}]$ is given in 
the basis $\mathbf{b}_L$ by the transpose of the matrix $M_{{\mathtt{n}}\text{-}\mathtt{i}}$ from \eqref{eq25}
(compare with \cite[Lemma~8]{AM}).

\begin{lemma}\label{lem22}
For $\mathtt{i},\mathtt{j}\in\mathbf{Q}$ we have
\begin{displaymath}
\mathrm{G}_{\mathtt{i}}\, P_{\mathtt{j}}\cong
\begin{cases}
P_{\mathtt{j}},& {\mathtt{i}}\neq {\mathtt{n}}\text{-}\mathtt{1},{\mathtt{j}}\text{-}\mathtt{1};\\
P_{{\mathtt{j}}\text{-}\mathtt{1}},& {\mathtt{i}}={\mathtt{j}}\text{-}\mathtt{1};\\
Y_{\mathtt{j}}, & {\mathtt{i}}={\mathtt{n}}\text{-}\mathtt{1},
\end{cases}
\end{displaymath}
where $Y_{\mathtt{j}}:=P_{\mathtt{j}}/P_{{\mathtt{n}}\text{-}\mathtt{1}}$.
\end{lemma}

\begin{proof}
The claim follows by applying $\mathrm{G}_{\mathtt{i}}$ to the following injective coresolution of $P_{\mathtt{j}}$:  
$P_{\mathtt{j}}\hookrightarrow I_{{\mathtt{n}}\text{-}\mathtt{1}}\tto I_{{\mathtt{j}}\text{-}\mathtt{1}}$ 
(here $I_{\mathtt{0}}:=0$).
\end{proof}

Choose in $\mathrm{Gr}(\mathcal{X})$ the basis $\mathbf{b}_P$ consisting of the classes of indecomposable projective 
modules $[P_{\mathtt{1}}]$, $[P_{\mathtt{2}}]$,\dots, $[P_{{\mathtt{n}}\text{-}\mathtt{1}}]$ 
(in this order). From Lemma~\ref{lem22} it follows that for $\mathtt{i}\neq {\mathtt{n}}\text{-}\mathtt{1}$ 
the linear transformation corresponding to the class $[\mathrm{G}_{\mathtt{i}}]$ is given in 
the basis $\mathbf{b}_P$ by the matrix $M_{{\mathtt{i}}+{\mathtt{1}}}$, while the linear 
transformation corresponding to the 
class $[\mathrm{G}_{{\mathtt{n}}\text{-}\mathtt{1}}]$ is given in the basis $\mathbf{b}_P$ by the 
following $(n\text{-}1)\times (n\text{-}1)$-matrix:
\begin{displaymath}
\left(
\begin{array}{cccccccccccc}
1&0&0&\dots&0&0&0\\
0&1&0&\dots&0&0&0\\
0&0&1&\dots&0&0&0\\
\vdots&\vdots&\vdots&\ddots&\vdots&\vdots&\vdots\\
0&0&0&\dots&1&0&0\\
0&0&0&\dots&0&1&0\\
-1&-1&-1&\dots&-1&-1&0\\
\end{array}
\right)  
\end{displaymath}
This gives an ``unusual'' effective representation of $\mathcal{C}_n$.

\subsection{Integral weightings in representations}\label{s4.2}

Let $z_1,z_2,\dots,z_{n-2}$ be positive integers. Consider the quiver
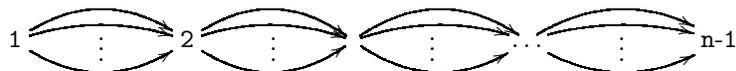
\begin{displaymath}
\xymatrix{
\mathtt{1}\ar@/^1pc/[rr]\ar@/^/[rr]\ar@{}[rr]|-{\vdots}\ar@/_1pc/[rr]
&&\mathtt{2}\ar@/^1pc/[rr]\ar@/^/[rr]\ar@{}[rr]|-{\vdots}\ar@/_1pc/[rr]
&&\ar@/^1pc/[rr]\ar@/^/[rr]\ar@{}[rr]|-{\vdots}\ar@/_1pc/[rr]&&
\dots\ar@/^1pc/[rr]\ar@/^/[rr]\ar@{}[rr]|-{\vdots}\ar@/_1pc/[rr]&&\mathtt{n}\text{-}\mathtt{1}
}
\end{displaymath}

\noindent
where we have $z_1$ arrows from $\mathtt{1}$ to $\mathtt{2}$, $z_2$ arrows from $\mathtt{2}$ to $\mathtt{3}$
and so on. All the above constructions and definitions carry over to this quiver in a straightforward way.
The only difference will be the explicit forms of matrices corresponding to the action of 
$[\mathrm{G}_{{\mathtt{i}}}]$. For example, in the basis of injective modules this action will be
given by the matrix
\begin{displaymath}
M_{\mathtt{i}}(z):=\left(
\begin{array}{cccccccccccc}
1&0&0&\dots&0&0&0&0&\dots&0&0&0\\
0&1&0&\dots&0&0&0&0&\dots&0&0&0\\
0&0&1&\dots&0&0&0&0&\dots&0&0&0\\
\vdots&\vdots&\vdots&\ddots&\vdots&\vdots&\vdots&\vdots&\ddots&\vdots&\vdots&\vdots\\
0&0&0&\dots&1&0&0&0&\dots&0&0&0\\
0&0&0&\dots&0&1&z_{i-1}&0&\dots&0&0&0\\
0&0&0&\dots&0&0&0&0&\dots&0&0&0\\
0&0&0&\dots&0&0&0&1&\dots&0&0&0\\
\vdots&\vdots&\vdots&\ddots&\vdots&\vdots&\vdots&\vdots&\ddots&\vdots&\vdots&\vdots\\
0&0&0&\dots&0&0&0&0&\dots&1&0&0\\
0&0&0&\dots&0&0&0&0&\dots&0&1&0\\
0&0&0&\dots&0&0&0&0&\dots&0&0&1
\end{array}
\right).  
\end{displaymath}
Similar modifications work also for the two other bases (of simple and indecomposable projective modules).
This should be compared with the weighted ``natural representation'' of a Hecke-Kiselman monoid constructed
in \cite{Fo}.

\subsection{Other Hecke-Kiselman semigroups}\label{s4.3}

Our construction generalizes, in a straightforward way, to all other Hecke-Kiselman semigroups associated to
acyclic quivers without $2$-cycles. However, Theorem~\ref{thmmain} does not hold in this generality. The reason
for this is the failure of Proposition~\ref{prop2} already in the case of the following quiver $\Gamma$:
\begin{displaymath}
\xymatrix{
\mathtt{1}\ar[rr]&&\mathtt{2}&&\mathtt{3}\ar[ll]
}
\end{displaymath}
In the general case the Grothendieck category of the corresponding $2$-category $\cG_{\Gamma}$ is identified with another,
more complicated, object. This object and its relation to the corresponding Hecke-Kiselman monoid will be 
studied in another paper.

\subsection{Koszul dual picture}\label{s4.4}

The category $\mathcal{A}_{\mathbf{Q}}^{\Bbbk}$ is positively graded in the natural way (the degree of each 
arrow equals $1$). This grading is Koszul and hence we can consider the Koszul dual category 
$(\mathcal{A}^!)_{\mathbf{Q}}^{\Bbbk}$ which is given by the opposite quiver 
\begin{displaymath}
\xymatrix{
\mathtt{1}&&\mathtt{2}\ar[ll]&&\mathtt{3}\ar[ll]&&\dots\ar[ll]&&\mathtt{n}\text{-}\mathtt{1}\ar[ll]
}
\end{displaymath}
together with the relations that each path of length two equals zero. The classical Koszul duality, 
see \cite{BGS}, provides an equivalence between bounded derived categories of graded 
$\mathcal{A}_{\mathbf{Q}}^{\Bbbk}$-modules and
graded $(\mathcal{A}^!)_{\mathbf{Q}}^{\Bbbk}$-modules. Note that all projective, injective and simple
$\mathcal{A}_{\mathbf{Q}}^{\Bbbk}$-modules are gradable and hence all our constructions have natural
graded lifts. In other words, $\cG_{n}$ has the natural positive grading in the sense of \cite{MM2}.
Using the theory of Koszul dual functors developed in \cite{Ma1,MSO}, one can reformulate the functors
$\mathcal{R}\mathrm{G}_{{\mathtt{i}}}$'s in terms of the derived category of 
$(\mathcal{A}^!)_{\mathbf{Q}}^{\Bbbk}$-modules.

\subsection{An alternative Koszul dual picture}\label{s4.5}

We can also try to consider the action of the usual projection functors 
$\mathrm{F}^!_{\mathtt{i}}$, ${\mathtt{i}}\in\mathbf{Q}$, 
on $(\mathcal{A}^!)_{\mathbf{Q}}^{\Bbbk}\text{-}\mathrm{mod}$. From \cite{Pa,Gr} we have that, mapping $\alpha_i$ 
to $\mathrm{F}^!_{{\mathtt{n}}\text{-}{\mathtt{i}}}$, 
${\mathtt{i}}\in\mathbf{Q}$, extends to a weak functorial action of $\mathcal{C}_n$
on $(\mathcal{A}^!)_{\mathbf{Q}}^{\Bbbk}\text{-}\mathrm{mod}$. Let $L^!_{\mathtt{i}}$ be the simple
$(\mathcal{A}^!)_{\mathbf{Q}}^{\Bbbk}$-module corresponding to ${\mathtt{i}}$ and 
$I^!_{\mathtt{i}}$ be the indecomposable injective with 
socle $L^!_{\mathtt{i}}$. Set $I^!:=\bigoplus_{{\mathtt{i}}\in\mathbf{Q}}I^!_{\mathtt{i}}$ and define
\begin{displaymath}
\mathrm{G}^!_{\mathtt{i}}:=
\mathrm{Hom}_{(\mathcal{A}^!)_{\mathbf{Q}}^{\Bbbk}}((\mathrm{F}^!_{\mathtt{i}}\,I^!)^* ,{}_-).
\end{displaymath}
Then $\mathrm{G}^!_{\mathtt{i}}$ is left exact and we have 
$\mathrm{G}^!_{\mathtt{i}}\vert_{\mathcal{I}_{(\mathcal{A}^!)_{\mathbf{Q}}^{\Bbbk}}}\cong
\mathrm{F}^!_{\mathtt{i}}\vert_{\mathcal{I}_{(\mathcal{A}^!)_{\mathbf{Q}}^{\Bbbk}}}$ 
(similarly to Subsection~\ref{s2.3}), where 
$\mathcal{I}_{(\mathcal{A}^!)_{\mathbf{Q}}^{\Bbbk}}$ denotes the category of injective 
$(\mathcal{A}^!)_{\mathbf{Q}}^{\Bbbk}$-modules. The first difference with  $\mathrm{G}_{\mathtt{i}}$ is failure of
Proposition~\ref{exact} for $\mathrm{G}^!_{\mathtt{i}}$ if $n\geq 4$ (note that 
$(\mathcal{A}^!)_{\mathbf{Q}}^{\Bbbk}\cong(\mathcal{A})_{\mathbf{Q}}^{\Bbbk}$ if $n=3$). For example, applying 
$\mathrm{G}^!_{\mathtt{4}}$ to the injective coresolution $L^!_{\mathtt{2}}\hookrightarrow  
I^!_{\mathtt{2}}\to I^!_{\mathtt{3}}\to I^!_{\mathtt{4}}\to \ldots \to I^!_{\mathtt{n-2}} \tto I^!_{\mathtt{n-1}}$
of $L^!_{\mathtt{2}}$, we get $\mathcal{R}^1\mathrm{G}^!_{\mathtt{4}}\, L^!_{\mathtt{2}}\cong L^!_{\mathtt{4}}\neq 0$.

This observation implies that (for $n\geq 5$)
\begin{equation}\label{eq57}
\mathcal{R}\mathrm{G}^!_{\mathtt{4}}\circ \mathcal{R}\mathrm{G}^!_{\mathtt{1}}\not\equiv
\mathcal{R}\mathrm{G}^!_{\mathtt{1}}\circ \mathcal{R}\mathrm{G}^!_{\mathtt{4}},
\end{equation}
that is the functors $\mathcal{R}\mathrm{G}^!_{\mathtt{4}}$ and $\mathcal{R}\mathrm{G}^!_{\mathtt{1}}$ 
do {\em not} satisfy the 
defining relations for $\mathcal{C}_n$. Indeed, for $n=5$ evaluating the right hand side of \eqref{eq57} at $I^!$ we get
\begin{displaymath}
\mathcal{R}\mathrm{G}^!_{\mathtt{4}}\, I^!\cong\mathrm{F}^!_{\mathtt{4}}\, I^!\cong 
I_{\mathtt{1}}^!\oplus I_{\mathtt{2}}^!\oplus I_{\mathtt{3}}^!.
\end{displaymath}
The result is injective and hence acyclic for $\mathrm{G}^!_1$, which means that the right hand side 
of \eqref{eq57} produces no homology in homological position $1$ when evaluated at $I^!$. On the other hand,
\begin{displaymath}
\mathcal{R}\mathrm{G}^!_{\mathtt{1}}\, I^!\cong\mathrm{F}^!_{\mathtt{1}}\, I^!\cong
L^!_{\mathtt{2}}\oplus I^!_{\mathtt{2}}\oplus I^!_{\mathtt{3}}\oplus I^!_{\mathtt{4}} 
\end{displaymath}
and this is not acyclic for $\mathrm{G}^!_{\mathtt{4}}$ by the computation above, which means that the left hand side 
of \eqref{eq57} produces a non-zero homology in homological position $1$ when evaluated at $I^!$. 

Similar arguments show the inequality for $n>5$, since we have   
\begin{displaymath}
\mathcal{R}\mathrm{G}^!_{\mathtt{i}}\, I^!\cong\mathrm{F}^!_{\mathtt{i}}\, I^!\cong  L^!_{\mathtt{i+1}}\oplus 
\bigoplus_{{\mathtt{i}\neq\mathtt{j}}\in\mathbf{Q}}I^!_{\mathtt{j}} 
\end{displaymath}
(where $L^!_{\mathtt{n}}:=0$) and $\mathcal{R}^1\mathrm{G}^!_{\mathtt{1}}\, L^!_{\mathtt{5}}\cong 0$. Hence the weak functorial action of $\mathcal{C}_n$ on $(\mathcal{A}^!)_{\mathbf{Q}}^{\Bbbk}\text{-}\mathrm{mod}$ does 
not seem to be naturally extendable to $\mathcal{D}^b((\mathcal{A}^!)_{\mathbf{Q}}^{\Bbbk})$.

Using spectral sequence arguments (see the remark after Lemma~\ref{lem1}) one proves the following relations:  
\begin{gather*}
\mathcal{R}\mathrm{G}^!_{\mathtt{i}}\circ \mathcal{R}\mathrm{G}^!_{\mathtt{i}}\cong 
\mathcal{R}\mathrm{G}^!_{\mathtt{i}},\quad\quad
\mathcal{R}\mathrm{G}^!_{\mathtt{i}}\circ \mathcal{R}\mathrm{G}^!_{{\mathtt{i}}+{\mathtt{2}}}\cong
\mathcal{R}\mathrm{G}^!_{{\mathtt{i}}+{\mathtt{2}}}\circ \mathcal{R}\mathrm{G}^!_{\mathtt{i}},\\
\mathcal{R}\mathrm{G}^!_{\mathtt{i}}\circ 
\mathcal{R}\mathrm{G}^!_{{\mathtt{i}}+{\mathtt{1}}}\circ \mathcal{R}\mathrm{G}^!_{\mathtt{i}}
\cong \mathcal{R}\mathrm{G}^!_{{\mathtt{i}}+{\mathtt{1}}}\circ\mathcal{R}\mathrm{G}^!_{\mathtt{i}} 
\circ \mathcal{R}\mathrm{G}^!_{{\mathtt{i}}+{\mathtt{1}}}.
\end{gather*}
It is an interesting question to determine exactly what kind of monoid (if any) the functors 
$\mathcal{R}\mathrm{G}^!_{\mathtt{i}}$'s generate.

\subsection{Combinatorics of subbimodules in $I$}\label{s4.6}

The set $\mathcal{C}_n$ has a natural partial order given by $f\leq g$ if and only if $f(i)\leq g(i)$ for all 
$i\in \mathbf{N}_{n}$.
Let $\mathfrak{I}$ denote the set of subbimodules in the 
$\mathcal{A}_{\mathbf{Q}}^{\Bbbk}\text{-}\mathcal{A}_{\mathbf{Q}}^{\Bbbk}$-bimodule $I$. Then $\mathfrak{I}$ 
is partially ordered  with respect to inclusions. Define a map $\Theta:\mathfrak{I}\to \mathcal{C}_n$,
$X\mapsto \Theta_X$, as follows: Let $X\in \mathfrak{I}$. As each $I_{\mathtt{i}}$ is uniserial and 
different $I_{\mathtt{i}}$'s have non-isomorphic socles, we have 
$X=\bigoplus_{{\mathtt{i}}\in\mathbf{Q}}X\cap I_{\mathtt{i}}$. If $I_{\mathtt{i}}=X\cap I_{\mathtt{i}}$, 
set $t_i:=0$. If $I_{\mathtt{i}}\neq X\cap I_{\mathtt{i}}$, let 
$t_i$ be such 
that $L_{{\mathtt{t}}_{\mathtt{i}}}$ is the socle of $I_{\mathtt{i}}/(X\cap I_{\mathtt{i}})$ 
(in particular, $t_i\leq i$). The right action on $X$ is given 
by surjections $\varphi_i:I_{{\mathtt{i}}+{\mathtt{1}}}\tto I_{\mathtt{i}}$. 
Then $\varphi_i(X\cap I_{{\mathtt{i}}+{\mathtt{1}}})\subset X\cap I_\mathtt{i}$ implies that 
$t_i\leq t_{i+1}$. Now define $\Theta_X$ as follows: 
\begin{displaymath}
\Theta_X(i):=
\begin{cases}
1,& i=1;\\
1+t_{i-1},& i\neq 1.
\end{cases}
\end{displaymath}
Then the above properties of $t_i$ imply that $\Theta_X\in \mathcal{C}_n$. 
The following is a dual version of \cite[Exercise~6.25(a)]{St} and \cite{CP}.

\begin{proposition}\label{prop23}
The map $\Theta$ is an isomorphism of partially ordered sets. 
\end{proposition}

\begin{proof} 
Injectivity of $\Theta$ follows directly from construction. Surjectivity follows from the observation that,
choosing any $t_i$'s satisfying $t_i\leq i$ we can define $X\cap I_{\mathtt{i}}$ as a unique submodule of
$I_{\mathtt{i}}$ such that $L_{{\mathtt{t}}_{\mathtt{i}}}$ is the socle of 
$I_{\mathtt{i}}/(X\cap I_{\mathtt{i}})$ and under the condition $t_i\leq t_{i+1}$ the space
$X:=\bigoplus_{{\mathtt{i}}\in\mathbf{Q}}X\cap I_{\mathtt{i}}$ becomes a subbimodule of $I$. 
That $\Theta$ is a homomorphism 
of posets is straightforward by construction.
\end{proof}

\subsection{Another interpretation of $1$-morphisms in $\cG_n$}\label{s4.7}

Proposition~\ref{prop23} allows for another interpretation of $1$-morphisms in $\cG_n$. For $f\in\mathcal{C}_n$
consider the subbimodule $X:=\Theta^{-1}(f)$ in $I$. Let $f=\alpha_{i_1}\alpha_{i_2}\cdots\alpha_{i_k}$ be some
decomposition of $f$ into a product of generators in $\mathcal{C}_n$. Set 
$\mathbf{i}:=(\mathtt{i}_1,\mathtt{i}_2,\dots,\mathtt{i}_k)$.

\begin{proposition}\label{prop24}
We have $\mathrm{G}_{\mathbf{i}}\cong \mathrm{Hom}_{\mathcal{A}_{\mathbf{Q}}^{\Bbbk}}((I/X)^*,{}_-)$. 
\end{proposition}

\begin{proof}
As $P^*$ is an injective cogenerator of $\mathcal{A}_{\mathbf{Q}}^{\Bbbk}\text{-}\mathrm{mod}$, 
it is enough to check that there is a natural isomorphism
$\mathrm{Hom}_{\mathcal{A}_{\mathbf{Q}}^{\Bbbk}}((I/X)^*,P^*)\cong \mathrm{G}_{\mathbf{i}}\, P^*$.
For the left hand side we have the natural isomorphism
$\mathrm{Hom}_{\mathcal{A}_{\mathbf{Q}}^{\Bbbk}}((I/X)^*,P^*)\cong I/X$ as shown in Subsection~\ref{s2.3}.

To compute the right hand side we recall that from the definition of $\mathrm{G}_{\mathbf{i}}$ we have
a natural transformation from the identity functor to $\mathrm{G}_{\mathbf{i}}$ which is surjective on
injective modules. Therefore it is enough to check that $\mathrm{G}_{\mathbf{i}}\, P^*\cong I/X$ as a left
module. Since $\mathrm{G}_{\mathbf{i}}$ maps injectives to injectives, it is enough to check that the
multiplicities of indecomposable injectives in $\mathrm{G}_{\mathbf{i}}\, P^*$ and $I/X$ agree. This follows
by comparing Theorem~\ref{thmmain}\eqref{thmmain.4} with the definition of $\Theta$ and 
Proposition~\ref{prop23}.
\end{proof}

\section{The category of $2$-morphisms and cell $2$-representations of $\cG_n$}\label{s5}

\subsection{The category of $2$-morphisms of $\cG_n$}\label{s5.1}

For every $f\in\mathcal{C}_n$ fix some indecomposable $1$-morphism $\mathrm{F}_f\in\cG_n$ such that 
$\Phi(f)=[\mathrm{F}_f]$ and set $X_f:=\Theta^{-1}(f)$. 

\begin{proposition}\label{prop71}
For $f,g\in \mathcal{C}_n$ we have $\cG_n(\mathrm{F}_f,\mathrm{F}_g)\cong 
\mathrm{Hom}_{\mathcal{A}_{\mathbf{Q}}^{\Bbbk}\text{-}\mathcal{A}_{\mathbf{Q}}^{\Bbbk}}(I/X_f,I/X_g)$.
The latter is nonzero if and only if $f\leq g$. If $f\leq g$, then 
$\mathrm{Hom}_{\mathcal{A}_{\mathbf{Q}}^{\Bbbk}\text{-}\mathcal{A}_{\mathbf{Q}}^{\Bbbk}}(I/X_f,I/X_g)$
is one-dimensional and is generated by the natural projection.
\end{proposition}

\begin{proof}
The isomorphism  $\cG_n(\mathrm{F}_f,\mathrm{F}_g)\cong 
\mathrm{Hom}_{\mathcal{A}_{\mathbf{Q}}^{\Bbbk}\text{-}\mathcal{A}_{\mathbf{Q}}^{\Bbbk}}(I/X_f,I/X_g)$ follows from 
Proposition~\ref{prop24}.
The simple top of the bimodule $I$ has composition multiplicity $1$ in $I$ and hence in all its non-zero quotients.
This implies that any non-zero map $I/X\to I/Y$ is a projection. Such a projection exists if and only if 
$X\subset Y$, which implies the rest of the proposition.
\end{proof}

\subsection{Cell $2$-representations of $\cG_n$}\label{s5.2}

The $2$-category $\cG_n$ does not have any weak involution and hence is not a fiat $2$-category in the sense of
\cite{MM1}. Nevertheless, we can still construct, by brute force, cell $2$-representations of $\cG_n$ which are
similar to cell $2$-representations of fiat categories constructed in \cite{MM1}. From 
Theorem~\ref{thmmain} we know that the multisemigroup $\mathcal{S}[\cG_n]$ is in fact a semigroup isomorphic to
$\mathcal{C}_n$. In particular, $\mathcal{S}[\cG_n]$ is a $\mathcal{J}$-trivial monoid. This means that all
left and all right cells of $\mathcal{S}[\cG_n]$ are singletons. 
For $X\subset\mathbf{Q}$ we denote $\varepsilon_X:=\alpha_{i_1}\alpha_{i_2}\cdots\alpha_{i_k}$, where we have
$X=\{i_1<i_2<\cdots<i_k\}$. Then $\{\varepsilon_X:X\subset\mathbf{Q}\}$ coincides with the set 
$E(\mathcal{C}_n)$ of all idempotents in $\mathcal{C}_n$. 

Consider the principal $2$-representation $\cG_n(\mathcal{X},{}_-)$ of $\cG_n$ and let 
$\overline{\cG_n(\mathcal{X},{}_-)}$ be its abelianization (in the sense of \cite{MM1,MM2}). Objects in 
$\overline{\cG_n(\mathcal{X},\mathcal{X})}$ are diagrams of the form $\beta:\mathrm{F}\to\mathrm{F}'$ where
$\mathrm{F},\mathrm{F}'$ are $1$-morphisms and $\beta$ is a $2$-morphism; morphisms in 
$\overline{\cG_n(\mathcal{X},\mathcal{X})}$ are usual commutative diagrams modulo right homotopy; 
and the $2$-action of $\cG_n$ is defined componentwise.

Consider the Serre subcategory $\mathcal{Z}$ of $\overline{\cG_n(\mathcal{X},\mathcal{X})}$ generated by all
simple tops of indecomposable projective objects $0\to \mathrm{F}_{f}$, where $f<\varepsilon_X$. As the set of all
$f$ satisfying $f<\varepsilon_X$ is an ideal in $\mathcal{C}_n$, from Proposition~\ref{prop71} it follows that 
$\mathcal{Z}$ is invariant under the $2$-action of $\cG_n$. Consider the abelian quotient 
$\overline{\cG_n(\mathcal{X},\mathcal{X})}/\mathcal{Z}$ with the induced $2$-action of $\cG_n$.

By Proposition~\ref{prop71}, the image of the indecomposable projective object $0\to \mathrm{F}_{\varepsilon_X}$
in this quotient is both simple and projective and hence its additive closure, call it $\mathcal{Q}_X$, is equivalent to 
$\Bbbk\text{-}\mathrm{mod}$. Similarly to the above, $\mathcal{Q}_X$ is invariant under the $2$-action of $\cG_n$
and hence has the structure of a $2$-representation of $\cG_n$ by restriction. This is the cell $2$-representation 
of $\cG_n$ associated  to the regular left cell $\{\mathrm{F}_X\}$. It is easy to see that for $i\in X$
the $1$-morphism $\mathrm{F}_{\alpha_i}$ acts on $\mathcal{Q}$ as the identity functor (up to isomorphism)
while for $i\not\in X$ the $1$-morphism $\mathrm{F}_{\alpha_i}$ acts on $\mathcal{Q}$ as zero.
All $2$-morphisms between non-isomorphic $1$-morphisms become zero in $\mathcal{Q}_X$.

\vspace{1mm}

\noindent
Anna-Louise Grensing, Bergische Universit{\"a}t 
Wuppertal, Gaussstrasse 20, 42097, Wuppertal, GERMANY,
{\tt grensing\symbol{64}math.uni-wuppertal.de}
\vspace{0.1cm}

\noindent
Volodymyr Mazorchuk, Department of Mathematics, Uppsala University,
Box 480, 751 06, Uppsala, SWEDEN, {\tt mazor\symbol{64}math.uu.se};
http://www.math.uu.se/$\tilde{\hspace{1mm}}$mazor/.
\vspace{0.1cm}

\end{document}